\documentclass[11pt]{article}
\usepackage{mylatexsetting}

\title{An Optimal Selection Problem Associated with the Poisson Process}
\author{Zakaria Derbazi \\{\it\small Queen Mary University of London}}
\date{\today}
\begin{document}
	\maketitle
	\noindent
	\begin{abstract}%
Cowan and Zabczyk (1978) introduced a continuous-time generalisation of the secretary problem, where offers arrive at epochs of a homogeneous Poisson process. We expand their work to encompass the last-success problem under the Karamata-Stirling success profile. In this setting, the $k$th trial is a success with probability $p_k=\theta/(\theta+k-1)$, where $\theta > 0$. In the best-choice setting ($\theta=1$), the myopic strategy is optimal, and the proof hinges on verifying the monotonicity of certain critical roots. We extend this crucial result to the last-success case by exploiting a connection to the sign of the derivative in the first parameter of a quotient of Kummer's hypergeometric functions. Additionally, we establish an Edmundson-Madansky inequality applicable to Poisson random variables. This result enables us to adopt a probabilistic approach to derive bounds and asymptotics of the critical roots. This strengthens and improves the findings of Ciesielski and Zabczyk (1979).
	\end{abstract}

	\section{Introduction}
	Cowan and Zabczyk (1978)  considered a continuous-time extension of the classical secretary problem. In this scenario, the task is to select the best apartment within a fixed time horizon, $T$. Apartments are presented at epochs of a homogeneous Poisson process with a known intensity $\lambda > 0$. The decision-maker may accept or discard each presented apartment irrevocably. Each arrival can be ranked relative to previously seen apartments, and all permutations of ranks are assumed equally likely and independent of the arrival process.
	
	The optimal strategy selects the apartment  presented at time $t$, whose index is $k$, provided that the following condition is satisfied
	\begin{alignat}{2}\label{cutoff}
		\lambda (T-t) \le a_k.
	\end{alignat}
	Each \textit{cutoff} $a_k$ ($k=1, 2, \ldots$) is the unique solution to the equation $h_k(x) = 0$, where $h_k(x)$ is formed by the difference between the probability of zero successes and exactly one success and is given by
	\begin{alignat}{2}\label{eqn.cz}
		h_k(x) \coloneqq \frac{1}{k} + \sum_{n=1}^\infty \frac{x^k}{n!(k+n)} \left(1-\sum_{j=1}^n \frac{1}{ j + k-1} \right).
	\end{alignat}
	For a strategy to be optimal, the sequence of roots $(a_k)_{k \ge 0}$ must be increasing. Cowan and Zabczyk developed an ingenious argument to establish this property. They begin by establishing the uniqueness of the roots,  leveraging the continuity of the function $h_k$ and all of its derivatives, along with the behaviour of the first and second derivatives,  in the vicinity of the root $a_k$. Subsequently, they exploit a relationship between the change in probability of exactly one success across consecutive trials and the difference between the probabilities of one and zero successes to prove monotonicity (see p. 588--590 in \cite{CZ}).

	Cowan and Zabczyk's optimal selection model can be categorised as a best-choice problem with a random number of observations, denoted  $N$, with an a priori Poisson distribution. These observations arrive uniformly and independently at times in $[0,T]$. 
	This formulation might appear similar to the best-choice problem introduced by Presman and Sonin (1972), in which $N$ is a Poisson-distributed random variable. However,  the key distinction lies in the role of time as a crucial decision factor \cite{ CZ, BrussCZ, Szajowski15, GD2}. Consequently, Cowan and Zabczyk's model is more accurately characterised as a best-choice problem with random observation times, where the number of items to be inspected has a Poisson prior distribution.
	
	We aim to extend this framework to encompass last-success problems, where the decision-maker only observes a sequence of indicator variables. In this setting, the defining characteristic of the trials (and, consequently, the problem itself) is the sequence of success probabilities, which we refer to as the \textit{success} profile or \textit{record} profile. In the classical best-choice problem, the profile is given by the success probability $p_k=1/k$  for the $k$th trial. 
	
	Our investigation is closely related to the theory of random records \cite{Gaver, Nevzorov}. Gaver (1976) played a pivotal role in linking this theory to optimal stopping problems. He examined the embedding of record times within a point process and analysed the problem of selecting the last record appearing between epochs of a Poisson process. Bruss (1988) further highlighted the strong relationship between optimal selection problems and optimal stopping related to record sequences. However, the initial work on last-success problems is attributed to Pfeifer (1989). He examined a sequence of record indicators and addressed the problem of selecting the last success within that sequence. For best-choice problems with a random number of trials,  a brief overview of recent advancements is provided in \cite{GD2}.
	
	A common strategy for solving last-success problems involves checking whether they fall within the \textit{monotone case} of optimal stopping theory \cite{CRS, FergusonBook}. Should this criterion be met, the optimal strategy is deduced via a nonincreasing sequence of cutoffs. These cutoffs, indexed by the number of arrivals, represent the earliest instant at which accepting the current offer (record) becomes optimal. The optimality of the myopic strategy is ensured whenever the sequence of cutoffs is nonincreasing.  
	
	Cowan and Zabczyk's problem has led to further research on optimal selection related to Poisson arrivals. Bruss (1987) broadened the framework to accommodate inhomogeneous Poisson processes. Ano and Ando (2000) considered a random availability attached to each offer. Schwalko and Szajowski (2003) explored the variant of the problem of selecting the best or second-best apartment. Subsequently, Szajowski (2007) introduced a game-theoretic version of the problem, thus offering a distinct perspective on the optimal selection strategy. Considering the last-success version of the problem, Hofmann (1997) investigated Poisson arrivals with successes appearing according to Nevzorov's $F^\alpha$ record profile. In this setting, success probabilities are given by $p_k = \alpha_k / (\alpha_1 + \cdots + \alpha_k)$ for some positive $\alpha's$. However,  Hofmann fell short of demonstrating the crucial property of monotonicity of the roots. In scenarios where trials arrive based on an inhomogeneous Poisson process, the instances of stopping at the last success, the last $\ell$th success, and any of the last $\ell$ successes were studied by Bruss (2000), Bruss and Paindaveine (2000), and Tamaki (2011) respectively.

	Using the Karamata-Stirling profile to model the success probabilities in our analysis is motivated by two key factors. First, it relates to the well-known Ewens Sampling Formula (ESF)  \cite{Crane}. The number of distinct components within the ESF framework can be expressed as a sum of independent indicator variables. Each of these variables follows the Karamata-Stirling profile, where the probability of success in the $k$thth trial is given by
	\begin{equation}\label{profile}
		p_k=\frac{\theta}{(\theta+k-1)}\,, ~~~k\geq 1,
	\end{equation}
	with  \textit{mutation} parameter     $\theta>0$. The case $\theta=1$ corresponds to the classical best-choice problem. 
	Second, the generating function of the number of successes in $n$ Bernoulli trials under the Karamata-Stirling profile exhibits advantageous analytical properties. This characteristic allows us to strengthen and simplify the solution presented by Cowan and Zabczyk while revealing a connection between last-success problems with random observation times and hypergeometric functions. In this regard, we rely on Kummer's confluent hypergeometric function properties to demonstrate that the sequence of cutoffs $(a_k)$  is increasing. This approach aligns with the recent work of \cite{GD1, GD2}, who employed the Gaussian hypergeometric function to analyse other power series distributions.

	The paper is structured as follows. Section 2 introduces the fundamental model based on a generic prior for the number of observations and a generic success profile. Subsequently, it derives the necessary condition for the optimality of the myopic strategy. Section 3 presents new findings related to Kummer's confluent hypergeometric function, focusing specifically on deriving an expression to determine the sign of the derivative in parameter of a quotient of Kummer's functions. Section 4 establishes the core optimality result, demonstrating the monotonicity of the cutoffs through a connection to the sign of the derivative in the first parameter of a quotient of Kummer's functions. Section 5 transitions to the scenario where observations are revealed at epochs of an inhomogeneous Poisson process. The optimal rule in the last-success problem is similar to the best-choice setting addressed by Bruss (1987), where arrival intensities are linked to the critical root. In the case of a continuous-time success probability function, the problem reduces to selecting the last arrival in an inhomogeneous Poisson process. This section concludes with an illustrative example that considers a scaling limit of the Karamata-Stirling profile as an example of a continuous-time success probability function. Subsequently,  the optimal threshold and the classical optimal winning probability are derived. Finally, Section 6 generalises the two theorems of Ciesielski and Zabczyk (1979) to encompass the Karamata-Stirling profile. We establish an Edmundson-Madansky inequality adapted to Poisson random variables using a coupling method and appropriate truncation. This crucial result enables us to refine the bounds on the roots and simplify the derivation of the asymptotic values.	
	\section{The Mathematical Model}
	
	Let $\left\{T_n, X_n\right\}_{n \geq 1}$ denote a marked Poisson process defined on a probability space $(\Omega, {\cal F}, \proba)$, where $T_1, T_2, \ldots$ are the chronologically ordered arrival epochs. The marks $\left\{X_n\right\}_{n \geq 1}$ form an i.i.d. sequence of Bernoulli random variables  independent of $\left\{T_n\right\}_{n \geq 1}$ having a success profile $\profile\coloneqq \left(p_n\right)_{n \ge 1}$, that is $\proba(X_n = 1) = p_n$.  Without loss of generality, assume that $0 < p_k < 1$ for all $k\ge1$.
	
	The bivariate process $\left\{T_n, X_n\right\}_{n \geq 1}$  is adapted to the natural filtration $({\cal F}_t)_{t\in[0,1]} \subseteq {\cal F}$, where the $\sigma$-algebra ${\cal F}_t$ encapsulates the complete knowledge of epochs and outcomes within the time interval $[0,t]$.  The random record process of points  $T_n$  at which record marks occur ($X_n=1$) can be regarded as a thinning of $\{T_n\}_{n \ge 1}$, where each mark $T_n$ is independently retained with probability $p_n$ \cite{BrowneBunge}. 
	
	Another way to model this process is to consider the marks $X_1, X_2, \ldots$ as a sequence of trials,  whose number $N$ is Poisson-distributed.  
	The number of successes among trials $k, k+1,\cdots,n$ has a  Poisson-binomial distribution with probability generating function
	$$z \mapsto \prod_{j=k}^n (1-p_j+zp_j).$$
	A simple computation gives the probability of zero successes and the probability of one success from stage $k \ge 1$ as
	\begin{alignat}{2}
		s_0(k, n)&\coloneqq\prod_{j=k}^n (1-p_j),~~~~k<n \label{s0},\\
		s_1(k, n)&\coloneqq s_0(k,n)\sum_{j=k}^n \frac{p_j}{1-p_j},  ~~~~~k < n \label{s0sm}.
	\end{alignat}
	
	Without loss of generality, we apply a deterministic time-change to transform arrivals from $[0, T]$ to $[0,1]$. Furthermore, given $N=n$, the trials arrive in continuous time at independent, uniformly distributed epochs $U_1, U_2, \cdots, U_n$ in [0,1].  We order these arrivals chronologically so each arrival epoch $T_k$ of the Poisson process in the original setting corresponds to the $k$th order statistic of $U_1,\cdots, U_n$.  We associate with each $T_k$, a trial $X_k$ resulting in success with probability $p_k$, independently of anything else, and we let $T_n$ undefined on the event $\{N<n\}$. The number of arrivals up to time $t$ is a mixed binomial process defined as
	$$
	N_t=\sum_{n=1}^{N}  \1{\{U_n\leq t\}}, ~~~t\in [0,1].
	$$

	We introduce the state variable $(t,k)\in [0,1]\times\{0,1,\cdots\}$ to represent the number of arrivals observed up to any epoch $t$ and denote the occurrence of exactly $k$ arrivals up to time $t$ by the event $\{N_t=k\}$. When the $k$th trial occurs at time $T_k = t$ and the outcome is a success, we designate a transition to this {\em success} state as $(t,k)^\circ$. 
	
	By  independence of the trials and the
	Markov property of  mixed binomial processes \cite{Kallenberg},
	the distribution of the point process of epochs and outcomes of the trials on $(t,1]$ both depend on ${\cal F}_t$ only through $N_t$.

	Let $\prior=(\pi_1, \pi_2, \ldots)$ be the prior distribution of the number of trials, where $\pi_n = \proba(N=n)$. The posterior distribution  factors as 
	\begin{alignat}{2}\label{posterior}
		\proba(N=k+j\,|\,N_t=k)&=\frac{\displaystyle\binom{k+j}{k}\pi_{k+j}(1-t)^jt^k}{\displaystyle\sum_{n=k}^\infty\binom{n}{k}\pi_{n}(1-t)^{n-k}t^k}\nonumberj
		&=f_k(t) {k+j\choose k} \pi_{k+j} (1-t)^j,~~~j\geq 0,
	\end{alignat}
	where $$f_k(t)= \left(\sum_{n=k}^\infty\binom{n}{k}\pi_{n}(1-t)^{n-k}\right)^{-1}$$ is a normalisation function.
	By (\ref{posterior}), the conditional probability that there are no successes following state  $(t,k)$ is given by
	\begin{equation}\label{S0}
		{\cal S}_0(t,k):=f_k(t) \sum_{j=0}^\infty \pi_{k+j} {k+j\choose k}  (1-t)^j s_0(k+1, k+j),
	\end{equation}
	which is the {\it adapted} reward  from stopping at a success state $(t,k)^\circ$.
	Analogously, the probability of exactly one success following state $(t,k)$  is
	\begin{equation}\label{S1}
		{\cal S}_1(t,k):=f_k(t) \sum_{j=m}^\infty \pi_{k+j} {k+j\choose k}  (1-t)^j s_1(k+1, k+j).
	\end{equation}
	Similarly, this probability represents the adapted reward from stopping at the next success state   $(t,k)^\circ$.

	The following result recalls some properties of ${\cal S}_1$ and  ${\cal S}_0$ that prove instrumental in the latter sections.
	\begin{lemma}[Lemma 2 in  \cite{GD2}] \label{lem.myopic} 
		For fixed $k\geq1$, suppose $\pi_k>0$. Then 
		\begin{itemize}
			\item[\rm(i)]  ${\cal S}_0(1,k)=1, ~{\cal S}_1(1,k)=0$,
			\item[\rm(ii)]   ${\cal S}_0(t,k)$ is nondecreasing in $t$, and is strictly increasing   if   $~\sum_{j=k+1}^\infty \pi_j>0$,
			\item[\rm(iii)]   		
			There exists cutoffs $a_{k}\in[0,1)$ such that 				$\sgn \left({\cal S}_1(t,k) - {\cal S}_0(t,k) \right) =\sgn (a_{k}-	t)$ for  $t\in(0,1]$.
		\end{itemize}
	\end{lemma}
	
	We define a {\it stopping strategy} to be  a random variable $\tau$ which takes  values in the random set of times $\{T_1,\cdots, T_N, 1\}$ and is adapted to ${\cal F}_t$ for  $t\in[0,1]$, where $\{\tau=1\}$ is interpreted as not stopping at all. Our analysis will focus solely on strategies stopping at success states $(T_k, k)^\circ$,  where the objective is to stop at the last success. For $k\geq 1$, stopping at  $(t,k)^\circ$  is regarded as a win if  $T_k=t$, and no successes occur in $(t,1)$.

	With every `stopping set'  $A_k\subset [0,1)\times{\mathbb N}$ we associate a {\it Markovian} stopping strategy which stops at the first success trial  $(t,k)^\circ$  with  $(t,k)$ falling in  $A_k$.
	The optimal strategy exists for arbitrary $\prior$ and $\profile$. Moreover, it is Markovian by a theorem on the exclusion of randomised stopping strategies (see Theorem 5.3 in \cite{CRS}).

	The {\it myopic} strategy $\tau^*$ is defined as the strategy stopping at the first state $(t,k)^\circ$ satisfying ${\cal S}_1(t,k) \le {\cal S}_0(t,k)$.
	This strategy is Markovian with  stopping set 
	\begin{equation*}
		A:=\bigcup\limits_{k=1}^\infty ([a_k,1)\times \{k\}),
	\end{equation*}
	where  the cutoffs  $(a_k)$ are defined in Lemma \ref{lem.myopic}. Each cutoff $a_k$ corresponds to the earliest time $\tau^*$ can accept a success trial with index $k$. 
	If  $(t,k)\notin A$, then stopping in the state  $(t,k)^\circ$ is not optimal.

	The so-called {\it monotone case}   of  the optimal stopping theory holds if  $A$ is 
	`closed',  meaning that after entering $A$, the sequence of successes  $(t,N_t)^\circ$  does not exit the set (see  \cite{CRS, FergusonBook}).
	This condition is sufficient for the optimality of  $\tau^*$.
	Since the sequence of successes is increasing in both components,
	this property of $A$ is equivalent to the condition
	$a_1\geq a_2\geq\cdots$, characterising the optimality of $\tau^*$.

	\section{Hypergeoemtrics}\label{HGF}
	Recall the definition of Kummer's hypergeometric function 
	\begin{equation}\label{Kummer}
		M(a,c,x):= \sum_{j=0}^\infty \frac{(a)_j }{(c)_j}  \frac{x^j}{j!}\,,
	\end{equation}
	viewed as a function of  real parameters $a, c$  and $x > 0$, where $c > a >0$, its basic integral representation is given by
	\begin{alignat}{2}
		M(a,c,x) \coloneqq \frac{\Gamma(c)}{\Gamma(a)\Gamma(c-a)}\int_{0}^1 e^{xu}u^{a-1}(1-u)^{c-a-1} \dx[u]\label{kummer.integral2},
	\end{alignat}
	where $\Gamma$ is the Gamma function, which is defined by Euler's integral of the second kind  as follows
	\begin{alignat}{2}\label{HGF.gamma}
		\Gamma(x) &\coloneqq \int_0^\infty u^{x-1}e^{-u} \dx[u]  ~~~~\text{for } x > 0.
	\end{alignat}

	Let $D_x$ denote the derivative with respect to the variable (or the parameter) $x$. The following elementary results are foundational for the proofs presented in this section.
	\begin{lemma}\label{lem.monotonicity.log}
		Consider two functions $f: I \to J$ and $g: I \to J$, where $I, J \subseteq (0, \infty)$. If $f$ and $g$ are differentiable,  then it holds that 
		\begin{alignat*}{2}
			\sgn D_x  \log f(x)  &= \sgn D_x f(x),\\[5pt]
			\sgn D_x \left\{f(x) g(x)\right\}  &= \sgn D_x  f(x) ~~~~ \text{ if } \sgn D_x  f(x)  = \sgn D_x  g(x),\\
			\sgn D_x \left\{\frac{f(x)}{g(x)}\right\} &= \sgn D_x  f(x) ~~~~ \text{ if } \sgn D_x  f(x)  = -\sgn D_x  g(x).
		\end{alignat*}
	\end{lemma}
	\begin{proof}
		Given that $f, g$ are strictly positive, it is evident that $$\sgn D_x \log f(x) \coloneqq \sgn \frac{ D_x f(x)}{f(x)}= \sgn  D_x f(x).$$
		From this result  and by definition of the $\log$ function, we have the elementary relations
		\begin{alignat*}{2}
			\sgn D_x \left\{f(x) g(x)\right\} =  \sgn D_x \log \left\{f(x) g(x)\right\}=  \sgn \left\{ D_x \log f(x) + D_x \log g(x)\right\}, 
		\end{alignat*}
		and 
		\begin{alignat*}{2}
			\sgn D_x\left\{{f(x)}/{g(x)}\right\}=  \sgn D_x \log \left\{{f(x)}/{g(x)}\right\}=  \sgn \left\{ D_x \log f(x) - D_x \log g(x)\right\}.
		\end{alignat*}
		The assertions are readily demonstrated by a simple sign analysis.
	\end{proof}

	The next result provides a test for assessing the sign of expressions of the kind
	$$D_x \frac{\Gamma(ax+x_1)}{\Gamma(ax+x_2)},$$
	for arbitrary real numbers $x_1, x_2$ and $a$.  We employ a lengthier but more instructive argument based on Euler's integral of the first kind. This approach circumvents the direct proof, which uses the logarithmic derivative of the Gamma function.
	\begin{lemma}\label{lem.gamma.sign}
		For a fixed $ c \in \mr$, the following sign identity holds:
		\begin{alignat*}{2}
			\sgn D_x \frac{\Gamma(x)}{\Gamma(x+c)} =-\sgn(c)  ~~~~~\text{ where }~~~~ x >0 \text { and }  x + c > 0.
		\end{alignat*}	
	\end{lemma}
	\begin{proof}
		The derivative of the quotient of Gamma functions is given by $$D_x \frac{\Gamma(x)}{\Gamma(x+c)} =  \frac{D_x \Gamma(x) \Gamma(x+c) -D_x \Gamma(x+c) \Gamma(x) }{\Gamma^2(x+c)}.$$
		We only need to check the derivative's numerator as the quotient's denominator is positive.
		Use Euler's representation of the Gamma function (\ref{HGF.gamma}) to note that 
		\begin{alignat*}{2}
			&D_x \Gamma(x) \Gamma(x+c) -D_x \Gamma(x+c) \Gamma(x)\\
			&=  \int_0^\infty \frac{1}{u}e^{-u}\log(u) \dx[u] \int_0^\infty v^{c-1}e^{-v} \dx[v]- \int_0^\infty v^{c-1}e^{-v}\log(v) \dx[v] \int_0^\infty \frac{1}{u}e^{-u} \dx[u]\\
			&=  \iint_{\overline{\mr}_+ \times \overline{\mr}_+} \left\{\frac{1}{uv}e^{-u-v}\right\}  v^{c}\log(u) \dx[u]\dx[v]-  \iint_{\overline{\mr}_+ \times \overline{\mr}_+} \left\{\frac{1}{uv}e^{-u-v}\right\}  v^{c}\log(v) \dx[u]\dx[v]\\ 
			&=  \iint_{\overline{\mr}_+ \times \overline{\mr}_+} \left\{\frac{u}{v}e^{-u-v}\right\}  v^{c}\log\left(\frac{u}{v}\right) \dx[u]\dx[v]\\ 
			&=  \iint_{0 < u < v < \infty} \left\{\frac{1}{uv}e^{-u-v}\right\}  v^{c}\log\left(\frac{u}{v}\right) \dx[u]\dx[v] +  \iint_{0 < v < u < \infty} \left\{\frac{1}{uv}e^{-u-v}\right\}  v^{c}\log\left(\frac{u}{v}\right) \dx[u]\dx[v]\\   
			&=  \iint_{0 < u < v < \infty} \left\{\frac{1}{uv}e^{-u-v}\right\}  v^{c}\log\left(\frac{1}{uv}\right) \dx[u]\dx[v] +  \iint_{0 < u < v < \infty} \left\{\frac{1}{uv}e^{-u-v}\right\}  u^{c}\log\left(\frac{v}{u}\right) \dx[u]\dx[v]\\   
			&=  \iint_{0 < u < v < \infty} \left\{\frac{1}{uv}e^{-u-v}\log\left(\frac{v}{u}\right) \right\} \left[u^{c} - v^{c}\right] \dx[u]\dx[v].
		\end{alignat*}
		The sign of the integrand when $u < v$ depends on the expression $u^{c} - v^{c}$ whose sign matches that of $-c$. This completes the proof.
	\end{proof}

	Next, we proceed to the main result of this section, which is to derive an expression for the sign of derivative in the first parameter and the second parameter, respectively, of a quotient of Kummer's functions.
	\begin{thm}\label{kummer.signn}
		For $ x > 0$,  the following sign identities hold
		\begin{alignat*}{2}
			{\rm(i)} ~\sgn D_a \displaystyle\left( \frac{M(a, c_1, x)}{M(a, c_2, x)} \right) = \sgn(c_1 - c_2), \text{ if } c_1 > a > 0 \text{ and } c_2> a > 0,\\[10pt]
			{\rm(ii)} ~\sgn D_c \displaystyle\left( \frac{M(a_1, c, x)}{M(a_2, c, x)} \right) = \sgn(a_2 - a_1), \text{ if } c > a_1 >0 \text{ and } c> a_2 > 0.
		\end{alignat*}
	\end{thm}
	
	\begin{proof}
		We only prove the first assertion of the theorem and defer part {\rm(ii)},  which is potentially of broader interest, to the appendix. Use part one of Lemma \ref{lem.monotonicity.log} to reformulate the analysis to the derivative of the log of a quotient instead since
		\begin{equation*}
			\sgn D_a \displaystyle\left( \frac{M(a, c_1, x)}{M(a, c_2, x)} \right) = 		\sgn D_a \displaystyle\log\left( \frac{M(a, c_1, x)}{M(a, c_2, x)} \right).
		\end{equation*}

		\noindent
		Now, let $\phi$  be this the integral part in Kummer's function representation (\ref{kummer.integral2}) 
		\begin{equation*}
			\phi(a,c,x)\coloneqq \int_{0}^1 e^{xu}\left( \frac{u}{1-u}\right)^{a-1} (1-u)^{c-2} \dx[u].
		\end{equation*}
		The logarithmic derivative in parameter $a$,	where $c_1, c_2$ are  constants satisfying  $c_1 > a > 0$ and $c_2 > a > 0$, respectively, is
		\begin{alignat}{2}
			D_a \displaystyle \log\left( \frac{M(a, c_1, x)}{M(a, c_2, x)} \right) &= D_a  \log M(a, c_1, x) - D_a  \log M(a, c_2, x)\nonumber\\
			&= D_a \left\{  \log \frac{\Gamma(c_1)}{\Gamma(c_2)} +\log\frac{\Gamma(c_2-a)}{\Gamma(c_1-a)} + \log \frac{\phi(a,c_1,x)}{\phi(a,c_2,x)}\right\}\nonumber\\
			&= D_a\log\frac{\Gamma(c_2-a)}{\Gamma(c_1-a)} + D_a \log \frac{\phi(a,c_1,x)}{\phi(a,c_2,x)}\label{proof.kummer.terms.check.Da}.
		\end{alignat}
		By the second part of  Lemma \ref{lem.monotonicity.log},  it is sufficient to show that the two logarithmic derivatives in (\ref{proof.kummer.terms.check.Da}) have the same sign. Using Lemma \ref{lem.gamma.sign}, the sign of the first quotient matches the sign of  $c_2-c_1$.
		
		For the second quotient,  check the sign of $ D_a \left\{ {\phi(a,c_1,x)}/{\phi(a,c_2,x)}\right\}$. Again, we only need to check the numerator's sign after taking the derivative. For positive constants, $c_1 > a > 0$ and $c_2 > a > 0$, use the Eulerian integral representation (\ref{kummer.integral2}) to express the derivative as a double integral in the following way
		\begin{alignat*}{2}
			&D_a \left\{ \frac{\phi(a,c_1,x)}{\phi(a,c_2,x)}\right\}\\
			&=  \int_{0}^1 e^{xu}\left( \frac{u}{1-u}\right)^{a-1} (1-u)^{c_1-2} \log\left( \frac{u}{1-u}\right)\dx[u]\int_{0}^1 e^{xv} \left( \frac{v}{1-v}\right)^{a-1} (1-v)^{c_2-2} \dx[v]
			\\&\qquad -\int_{0}^1 e^{xu}\left( \frac{u}{1-u}\right)^{a-1} (1-u)^{c_1-2} \dx[u]\int_{0}^1 e^{xv}\left( \frac{v}{1-v}\right)^{a-1} (1-v)^{c_2-2} \log\left( \frac{v}{1-v}\right)\dx[v]\\
		&=  \iint_{[0,1]\times[0,1]} e^{x(u+v)}\left( \frac{uv}{(1-u)(1-v)}\right)^{a-1} [(1-u)(1-v)]^{c_1-2} (1-v)^{c_2-c_1}  \log\left( \frac{u}{1-u} \right)\dx[u]\dx[v]\\
		&\qquad -\iint_{[0,1]\times[0,1]} e^{x(u+v)}\left( \frac{uv}{(1-u)(1-v)}\right)^{a-1} [(1-u)(1-v)]^{c_1-2} (1-v)^{c_2-c_1}  \log\left( \frac{v}{1-v} \right)\dx[u]\dx[v]\\
		&=\iint_{[0,1]\times[0,1]} \left\{\underbrace{ e^{x(u+v)}\left( \frac{uv}{(1-u)(1-v)}\right)^{a-1} [(1-u)(1-v)]^{c_1-2}}_{\text{symmetric in } u, v}\right\} (1-v)^{c_2-c_1}  \log\left( \frac{u(1-v)}{v(1-u)} \right)\dx[u]\dx[v]\\
		&=\iint_{0 < v < u < 1} \left\{\cdots\right\} (1-v)^{c_2-c_1}  \log\left( \frac{u(1-v)}{v(1-u)} \right)\dx[u]\dx[v]+ \iint_{0 < u < v < 1} \left\{\cdots\right\} (1-v)^{c_2-c_1}  \log\left( \frac{u(1-v)}{v(1-u)} \right)\dx[u]\dx[v]\\
		&=\iint_{0 < v < u < 1} \left\{\cdots\right\} (1-v)^{c_2-c_1}  \log\left( \frac{u(1-v)}{v(1-u)} \right)\dx[u]\dx[v]+ \iint_{0 < v < u < 1} \left\{\cdots\right\} (1-u)^{c_2-c_1}  \log\left( \frac{v(1-u)}{u(1-v)} \right)\dx[u]\dx[v]\\
		&=\iint_{0 < v < u < 1} \left\{\cdots\right\} \log\left( \frac{u(1-v)}{v(1-u)} \right) \left\{ (1-v)^{c_2-c_1}   - (1-u)^{c_2-c_1}  \right\}\dx[u]\dx[v].
	\end{alignat*}
	In the given domain of integration where $u>v$, the expression $\displaystyle \log\left( \frac{u(1-v)}{v(1-u)} \right)$ is positive. This means that the sign of the last integrand depends solely on $ (1-v)^{c_2-c_1}   - (1-u)^{c_2-c_1} $. But, the sign of this expression is precisely that of $c_2 - c_1$ as $1-v > 1-u$.
	
	Since the sign of both terms in the right-hand side of (\ref{proof.kummer.terms.check.Da}) coincides, we can conclude, through part two of Lemma \ref{lem.monotonicity.log}, that the sign of $D_a \displaystyle \log\left( \frac{M(a, c_1, x)}{M(a, c_2, x)} \right)$ matches that of $c_2-c_1$, which is the desired result.

\end{proof}

\section{The Poisson Prior}\label{section.homogeneous}
We focus our investigation to the case where $N$ follows a Poisson prior distribution $\prior\coloneqq (\pi_1, \pi_2, \ldots)$ with intensity $\lambda > 0$, where
\begin{equation*}
	\pi_n=     \frac{ \lambda^n}{n!}e^{-\lambda} ~~~~n\geq 0.
\end{equation*}
This formulation recovers the original scenario where the underlying pacing process $(N_t,~t\in[0,1])$  is Poisson. By the stationarity of  increments of the Poisson process,  the posterior of $N-N_t$ given $N_t=k$ is again Poisson with a rescaled intensity 
\begin{equation}\label{x=(1-t)}
	x=(1-t)\lambda.
\end{equation}
That is, 
\begin{alignat*}{2}
	\proba(N=k+j\,|\,N_t=k)&=\frac{x^{j} }{j!}e^{-x}.
\end{alignat*}
Based on this posterior distribution, the probability-generating function of the number of successes following state $(t,k)$  becomes
\begin{alignat*}{2}
	z\mapsto  e^{-x}\sum_{j=0}^\infty \frac{(k+\theta z)_j}{(k+\theta)_j}\, \frac{x^j}{j!}
	&=e^{-x} M(k + \theta z; k + \theta; x) \nonumberj
	&=M(\theta- \theta z; k + \theta; -x).
\end{alignat*}
where $M(a;c;x)$ is Kummer's hypergeometric function defined in (\ref{Kummer}). The last equality is a consequence of Kummer's transformation formula, where
\begin{alignat*}{2}
	M(a, c; x) &= e^xM(c-a, c; -x).
\end{alignat*}
The derivative of the generating function is given by
\begin{alignat*}{2}
	z \mapsto \theta e^{-x} D_a M(k + \theta z ,\theta+k, x).
\end{alignat*}
Explicit expressions for (\ref{S0}) and (\ref{S1}), which represent the probability of zero successes and the probability of exactly one success after state $(t,k)$, are as follows
\begin{align*}
	{\cal S}_0(1-x/\lambda,k)&=e^{-x} M(k;\theta+k;x),\nonumberj
	{\cal S}_1(1-x/\lambda,k)&=\theta e^{-x} D_a M(k;\theta+k;x).
\end{align*}
As ${\cal S}_0$ and $	{\cal S}_1$ are power series in $x$  and both series converge for all values of $x$,  the cutoffs, defined by Lemma \ref{lem.myopic}, can be recast in terms of a sole sequence of critical roots $(\gamma_k)$ of  ${\cal S}_1(1-x/\lambda,k)- {\cal S}_0(1-x/\lambda,k)=0$,  expressed through the equation
\begin{equation}\label{root.eqn.poi}
	\theta D_a M(k;\theta+k;x)  -  M(k;\theta+k;x) = 0.
\end{equation}
In the special case where  $\theta=1$, the equation reduces to that of Cowan and Zabczyk (\ref{eqn.cz}). 
Observe that the roots $(\gamma_k)$  do not depend on the intensity $\lambda$,  and are related to the cutoffs through the identity
\begin{equation}
	a_k=\left(1-\frac{\gamma_k}{\lambda} \right)_+.\label{a-alpha}
\end{equation}
The optimality of the myopic strategy for all finite $\lambda > 0$ is equivalent to the condition  $\gamma_k\uparrow$. That is, the roots must be nondecreasing with $k$. Notably, Cowan and Zabczyk posed and ingeniously solved this problem in the classical success profile setting ($\theta=1$). The argument for uniqueness relies on showing that each derivative of $h_k(x) \coloneqq {\cal S}_1(1-x/\lambda,k) -  {\cal S}_0(1-x/\lambda,k)$ has a unique root. The uniqueness is achieved by exploiting the continuity of this function and its derivatives. They approach the problem inductively by starting at higher derivatives and continuing iteratively to lower ones until $h_k(x)$ is reached. They demonstrated the monotonicity of roots based on a relation between $h_{k+1}(x)$  and the derivative of an expression involving ${\cal S}_1(1-x/\lambda,k) -  {\cal S}_1(1-x/\lambda,k+1)$. 
In contrast, our analysis of Cowan and Zabczyk's problem reveals a distinct structure characterised by a prominent connection to Kummer's hypergeometric function. The following theorem strengthens their original result by utilising a more general success profile and provides a more straightforward proof that capitalises on the properties of Kummer's hypergeometric function.
\begin{thm}\label{main.Poi}
	The roots $\gamma_k$ are increasing. Therefore, the stopping rule $\tau^*$ which calls for stopping at state $(t, k)^\circ$, where
	\begin{equation}\label{zabczyk.inequality}
		(1-t)\lambda \le \gamma_k,
	\end{equation}
	is optimal.	
\end{thm}
\begin{proof}
	First, we need establish that   the roots $\gamma_{k}$ of equation (\ref{root.eqn.poi}) are increasing and hence $\gamma_{k+1}>\gamma_{k}$ for all $k \ge 1$. It is sufficient to show that 
	$$D_a \log (M(k,\theta+k,x)) < D_a \log (M(k,\theta+k+1,x)), ~~~x\in(0, \infty).$$
	But this holds by the first identity of Lemma \ref{kummer.signn}.
	Consequently, by assertion {\rm(iii)} of Lemma \ref{lem.myopic},   it is optimal to stop at time $t \ge a_k$, where the cutoff $a_k$ is given by (\ref{a-alpha}). Using the definition of the cutoff, we obtain
	$t \ge 1-\gamma_k/\lambda$. Rearranging this inequality yields the desired result.
\end{proof}
\begin{rem}
	In the setting where $\lambda$ is  random with a Gamma prior distribution, $(N_t,~t\in[0,1])$  becomes a negative binomial process (mixed Poisson). This scenario has been addressed under different guises in \cite{Gaver,  OB, Kurushima, GD2}.
\end{rem}

\section{The Inhomogeneous Case}

We now turn to the case where the counting process $(N_t)$ is an inhomogeneous Poisson process with known time-dependent intensity $\lambda(t)$, where the function
\begin{equation}\label{number.arrivals}
	\Lambda(s) = \int_{0}^{s} \lambda(u) \dx[u],
\end{equation}
is continuous and $s \in [0,1]$.

In the best-choice case, Theorem 3 in \cite{BrussCZ} established the optimal strategy for selecting the best apartment, where arrivals are governed by $(N_t)$. The extension of this result to the last-success setting is immediate.
\begin{thm}
	Consider the last success problem with success profile (\ref{profile}), where trials arrive according to an inhomogeneous Poisson process $(N_t,~t \in [0,1])$ with mean arrivals given by (\ref{number.arrivals}). The optimal stopping strategy for selecting the last success is to stop at state $(t^*, k)^\circ$ such that $$\Lambda(1)-\Lambda(t^*) \le \gamma_k,$$ where $\gamma_k$ is the root of equation (\ref{root.eqn.poi}).
\end{thm}
\begin{proof}
	Given $N_1=n$, the trials arrive at independent  epochs $U_1, U_2, \cdots$ in [0,1]. These arrival times have a common distribution function $F$, where 
	\begin{equation}\label{arrival.time.dist}
		F(s)= \frac{\Lambda(s)}{\Lambda(1)}. 
\end{equation}
The points $F(U_1), F(U_2), \ldots$ are renewal times of a homogeneous Poisson process  in $[0,1]$ with known intensity $\Lambda(1)$. Therefore, by Theorem \ref{main.Poi}, the  myopic strategy $\tau^*$ is optimal, and the optimal rule is to stop at the success state $(t^*,k)^\circ$ satisfying 
\begin{equation}\label{zabczyk.inequality2}
	\left[ 1-F(t^*)\right]\Lambda(1) \le \gamma_k,
\end{equation}
which is a straightforward generalisation of (\ref{zabczyk.inequality}). Plugging (\ref{arrival.time.dist}) into (\ref{zabczyk.inequality2}) completes the proof.
\end{proof}

\subsection{The Case of the Poisson Process of Successes}
Let $p(t)$ denote the success probability of a trial occurring at time $t$. A success at time $t$ is independent of the arrival process $(N_t)$ and previous outcomes. Our objective is to maximise the probability of stopping at the last arrival time that coincides with the occurrence of a success.

Fix a positive integer $m$ and let  $\delta_m \coloneqq \{ 0=t_0 < t_1 < \cdots < t_m = 1\}$ be an partition of $[0,1]$. Denote the length of the $k$th subinterval $[t_{k-1}, t_k]$ by $\Delta t_{k}$. Let $p_k$ be the probability of a success occurring in the interval $(t_{k-1}, t_k]$. From the independence of successes and the properties of the Poisson process, it follows that
$$p_k = p(t_k)\lambda(t_k)\Delta t_{k} + o(\Delta t_{k}).$$
As $\norm{\delta_m} \coloneqq\max\limits_{1 \le i \le m} \Delta t_i$ tends to 0, the well defined function
\begin{equation}\label{success.intensity}
s(t_k) \coloneqq \lim_{\Delta t_k \to 0} \frac{p(t_k)\lambda(t_k)\Delta t_{k} + o(\Delta t_{k})}{\Delta t_{k}} = p(t_k)\lambda(t_k) .
\end{equation}
is referred to as the limiting \textit{success intensity} by Bruss \cite{Odds}.  This function represents a random record process of points  $\lambda(t)$ at which a record occurs with probability $p(t)$. This thinning of $(N_t)$  preserves the Poisson process property, where each point is independently retained with probability $p(t)$. We denote this new thinned Poisson process of \textit{successes} by $(P_t,~ t \in [0,1])$. Each arrival generated by $(P_t)$, whose intensity is  $s(t)$, corresponds to a success event. 

Let 
\begin{equation*}
\begin{dcases}
	m(t) \coloneqq \mE{P_t}= \int_0^t s(u) \dx[u], \text { and }\\
	M(t_0,t_1) \coloneqq m(t_1) -m(t_0)= \int_{t_0}^{t_1} s(u) \dx[u],
\end{dcases}
\end{equation*}
denote respectively  the expected number of arrivals of $(P_t)$ up to time $t$  and the expected number of arrivals in the time interval $[t_0, t_1]$, where $0 \le t_0 < t_1 < 1$.

In the context of the last success problem in continuous time, the optimal stopping rule $\tau^*$ translates to stopping at the last arrival time of the inhomogeneous process $(P_t)$. This case is much simpler than Cowan and Zabczyk's problem, given that all the arrivals are successes. Consequently,  it is optimal to stop on the next arrival of $(P_t)$, if any,  which occurs after time $t^*$, where
\begin{equation*}
t^*\coloneqq \sup\bigg\{ 0, ~\sup \big\{ t \in [0,1]: ~M(t, 1) \ge 1\big\}\bigg\}, \text{ and } \sup \{ \varnothing\} = -\infty.
\end{equation*}%

The threshold $t^*$ corresponds to the last time the expected number of future arrivals of $(P_t)$ is at most 1. The optimal stopping rule $\tau^*$ can be readily extended to the problem of stopping at the last $\ell$th success \cite{BrussPaindaveine, D1}. This generalisation follows intuitively where stopping is optimal after time $t^\star$ given by 
\begin{equation}\label{stopping.rule.ell}
t^\star \coloneqq \sup\bigg\{ 0, ~\sup \big\{ t \in [0,1]: ~M(t, 1) \ge \ell\big\}\bigg\}.
\end{equation}
In what follows, we present a simple example that illustrates a last-success problem which combines a continuous-time Karamata-Stirling profile with an inhomogeneous arrival process. 

\paragraph{Example.} Suppose that the mutation parameter of the  Karamata-Stirling profile exhibits a linear relationship with the number of trials $n$ such that a finite \textit{mutation rate} $\kappa \in (0, \infty)$ exists. In this case
\begin{equation}\label{mutation.rate}
\kappa \coloneqq \lim_{n \to \infty} \frac{\theta}{n}.
\end{equation}
To introduce a continuous-time adaptation of the Karamata-Stirling profile, divide both the numerator and denominator of (\ref{profile}) by $n$ and let $t = {(k-1)}/{n}$, where  $k=1, \ldots, n, $ be the proportion of trials elapsed. As  $n \to \infty$, we obtain a scaling limit of the Karamata-Stirling profile, denoted by $p(t)$, where
\begin{alignat*}{2}
p(t) &\coloneqq  \lim_{n \to \infty} \frac{\theta/n}{\theta/n + (k-1)/n}\nonumber\\
&=\frac{\kappa}{\kappa + t}, ~~ t \in [0,1].
\end{alignat*}
With this continuous-time Karamata-Stirling profile established, we shift our focus to analysing the arrival process governed by the following  intensity function 
$$\lambda(t) = (t + \kappa)^{\alpha}, $$
where $\alpha \in \mr$ is a fixed number  and $\kappa$ is the mutation rate defined by (\ref{mutation.rate}). Clearly, the success intensity (\ref{success.intensity}) becomes
\begin{equation*}
s(t) = \kappa (t+\kappa)^{-1+\alpha}.
\end{equation*}
When $\alpha = 0$, the original  arrival process ($N_t$) is homogeneous and $(P_t)$, the Poisson process of successes,  is inhomogeneous as $s(t) = p(t)$. The expected number of  arrivals (successes)  between  $0 \le t_0 < t_1$ and $t_1 < 1$ is given by
\begin{equation}\label{intensity.M}
M(t_0,t_1) = \kappa\log\left(\frac{t_1 + \kappa}{t_0 + \kappa}\right).
\end{equation}
Using (\ref{stopping.rule.ell}),  the optimal rule for selecting the last $\ell$th success, where $\ell \ge 1$ is to stop on the first arrival occurring after time $t^\star$, where
\begin{equation*}
t^\star = \sup\bigg\{ 0, ~\sup \big\{ t \in [0,1]: ~t \le (\kappa+1)e^{-^{\ell/\kappa}}-\kappa \big\}\bigg\}.
\end{equation*}
Next, recall that $P_1 - P_t$ has a Poisson distribution with intensity $M(t,1)$, therefore use (\ref{intensity.M}) to express the probability of $\ell$ arrivals as
\begin{equation}\label{poisson.prob}
\proba[P_1 - P_t = \ell] = \frac{1}{\ell!}{\left(-\kappa \log\left( \frac{t + \kappa}{1 + \kappa} \right)\right)^\ell}\left(  \frac{t + \kappa}{1 + \kappa} \right)^{\kappa}.
\end{equation}
Setting $t = (\kappa+1)e^{-^{\ell/\kappa}}-\kappa$ in (\ref{poisson.prob}),  we achieve the classical optimal winning probability $e^{-\ell}{\ell^\ell}/{\ell!}$ for stopping at the last $\ell$th success \cite{BrussPaindaveine, D1}.  This result translates to the ubiquitous winning probability $1/e$ in the last success setting ($\ell=1$).

Now suppose that $\alpha \ne 0$. An easy calculation gives the  expected number of arrivals of $(P_t)$ between times $t_0 \in [0,t_1)$ and $t_1  \in(0,1)$ as
\begin{alignat*}{2}
M(t_0, t_1) = \frac{\kappa}{\alpha}  \left\{  (t_1+\kappa)^{\alpha} - (t_0+\kappa)^{\alpha}\right\}.
\end{alignat*}
Consequently,  the optimal threshold becomes, in this case
$$	t^\star \coloneqq \sup\bigg\{ 0, ~\sup \big\{ t \in [0,1]: ~t \le \left( (1+\kappa)^\alpha - \frac{\alpha \ell}{\kappa} \right)^{\displaystyle\frac{1}{\alpha}} - \kappa \big\}\bigg\},$$
and the optimal winning probability is once again $e^{-\ell}{\ell^\ell}/{\ell!}.$\\

\noindent In the next section, we revert to the homogenous Poisson setting with a discrete Karamata-Stirling profile. Our objective is  to assess the asymptotic properties of the roots obtained in Section \ref{section.homogeneous}.

\section{Asymptotics and Bounds}

Unlike other priors, such as the negative binomial addressed in \cite{GD2}, the roots $\gamma_k$ do not converge. There is no Poisson process with infinite prior, which ensures that only finitely many nonzero cutoffs exist for each $\lambda \in (0,  \infty)$.  \cite{CZ,Ciesielski,  BG, BrussCZ} observed that the analogues of $\gamma_k$'s  do not accumulate, rather grow  linearly with $k$.

In the best-choice setting, where $\theta=1$, \cite{CZ} showed that $\gamma_k \sim k$ for large $k$ and \cite{Ciesielski} derived the bounds 
$$(e-1)(k-1) \le \gamma_k \le 	4e\left[ (e-1)k+1\right],$$
and asymptotics
$$\lim_{k \to \infty} \frac{\gamma_k}{k}   = (e-1). $$

We extend these findings to the last success setting under the Karamata-Sterling profile in our upcoming theorem. While we slightly improved the bounds on the roots compared to the best choice case  ($\theta=1$), the distinctive value of our work lies in the methodology employed. We primarily relied on probabilistic tools, which ultimately led to simplifying the derivation of the asymptotic values.
\begin{thm}[Main Theorem]\label{thm.bounds.roots}
For $k \ge 2$, the roots $(\gamma_k)$ of the fundamental equation (\ref{root.eqn.poi}) satisfy
\begin{equation}\label{poisson.root.bounds}
	(e^{\textstyle\frac{1}{\theta}+\frac{1}{2(k-1)}}-1)(k-1) \le \gamma_k  \le (e^{\textstyle\frac{1}{\theta}+\frac{1}{(e-1)\,\theta}}-1)(k-1).
\end{equation}
Moreover,
$$\lim_{k \to \infty} \frac{\gamma_k}{k}   = (e^{1/\theta}-1). $$
\end{thm}

\noindent
To prove this theorem, we require new notation and several supporting results; some of which are of independent interest.  We begin by defining two functions. The first is a ratio of two Pochammer symbols and the second is a generalised harmonic number.
\begin{equation}\label{def.varphi}
\begin{dcases}
	\phi(n, k, \theta)&\coloneqq \frac{(k)_n}{(k+\theta)_n}, ~~~ \phi(0, k, \theta)=1, \\
	\varphi(x, k) &\coloneqq   \psi(x + k ) -\psi(k), ~~~
\end{dcases}
\end{equation}
where $n, k$ are nonnegative integers,  $x$ is a nonegative real number,  $\theta$ is a positive real number and $\psi$ is the logarithmic derivative of the Gamma function. For brevity, we will often omit  the parameters $k$ and $\theta$ when they are clear from the context. Thus,  $\phi (n)$ and $\varphi(x)$ stand for $\phi(n, k, \theta)$ and $\varphi(x, k)$ respectively. 
\vskip 5pt
\noindent
Our first lemma highlights some features of the functions $\phi$ and $\varphi$.
\begin{lemma}\label{lem.convex.concave.varphi}
	Given a fixed nonnegative integer  $k$ and a positive real number $\theta$, the functions $\varphi(\cdot, k, \theta)$ and $\phi(\cdot, k)$ have the following properties:
	\begin{itemize}
		\item [\rm(i)] $\phi(\cdot, k, \theta)$ is log-convex and decreasing.
		\item [\rm(ii)] $ 	\varphi(\cdot, k) $ is concave and increasing.
	\end{itemize}
	
\end{lemma}
\begin{proof}
	Fix $\theta > 0$ and $k \in \{0,1, \ldots\}$. 
	To prove the first assertion, it suffices to show that for any nonnegative integer $n$, we have
	$ \phi^2(n, k, \theta) \le \phi(n+1,k, \theta)\phi(n-1,k, \theta) $.  The case $k=0$ is obvious. When $k > 0$, the result is easily seen by noting that
	\begin{alignat*}{2}
		\frac{\phi(n, k, \theta)}{\phi(n+1, k, \theta)} - \frac{\phi(n-1, k, \theta)}{\phi(n, k, \theta)} & = 
		\frac{(k)_n / (k+\theta)_n}{ (k)_{n+1} / (k+\theta)_{n+1}} -\frac{(k)_{n-1} / (k+\theta)_{n-1}}{ (k)_{n} / (k+\theta)_{n}}\\
		&=\frac{\theta}{k+n}-\frac{\theta}{k+n-1} < 0.
	\end{alignat*}
	The monotonicity of $\phi$ in the first parameter is clear from the ratio
	\begin{alignat*}{2}
		\frac{\phi(n, k)}{\phi(n+1, k)} &= \frac{(k)_n / (k+\theta)_n}{ (k)_{n+1} / (k+\theta)_{n+1}}
		&=\frac{k + \theta + n}{ k +n} > 1.
	\end{alignat*}	
	For part {\rm(ii)}, the concavity of $\varphi$ and its monotonicity follow from that of $\psi$. 
	This completes the proof
\end{proof}

\vspace*{5pt}
\noindent Our second lemma establishes bounds for the function $\varphi$. 
\begin{lemma}\label{lem.bounds.varphi}
	The relation
	\begin{alignat}{2}\label{varphi.ineq}
		\log\left(\frac{x+ k-1}{k-1} \right) - \frac{1}{2(k-1)} < \varphi (x, k)  <  \log\left(\frac{x + k}{k-1} \right),
	\end{alignat}	
	holds for  $k \in \{ 2, 3, \ldots\}$ and $x \in [0, \infty)$.
\end{lemma}
\begin{proof}
	Recall that 
	\begin{alignat*}{2}
		\frac{x}{x+1} &< \log(1+x) < x, &  x > 0,\phantom{ \text{  and }}\\
		\psi(x+1) &= \psi(x) + \frac{1}{x}& \quad x > 0, \text{ and }\\
		\log(x) -\frac{1}{x} &\le \psi(x) \le \log(x) -\frac{1}{2x},  &\quad x > 1 \qquad & \text{(relation 3.6.55 in \cite{Mitrinovic})}.
	\end{alignat*}
For the upper bound, 	
	\begin{alignat*}{2}
		\varphi(x, k)&\coloneqq \psi(x + k) - \psi(k-1) - \frac{1}{k-1} \\
		&\le \log(x + k) -\frac{1}{2(x+k)} - \log(k-1) \nonumberj
		&<   \log\left(\frac{x + k}{k-1} \right).
	\end{alignat*}	
The lower bound is derived in the same way
	\begin{alignat*}{2}
		\varphi(x, k) &\coloneqq \psi(x + k-1) + \frac{1}{x+k-1} - \psi(k-1) - \frac{1}{k-1} \\
		&\ge \log\left(\frac{x + k-1}{k-1} \right) +\frac{1}{2(k-1)}- \frac{1}{k-1}  \nonumberj
		&\ge   \log\left(\frac{x + k-1}{k-1} \right) - \frac{1}{2(k-1)}.
	\end{alignat*}	
		which is precisely our claim.
\end{proof}
\begin{rem}
In the case of an integer parameter $m > 0$ for $\varphi$, we may consider the elementary relation
$$
\log(j+1)-\log(j)=\int_j^{j+1} \frac{\dx[t]}{t}<\frac{1}{ j}<\int_{j-1}^j \frac{\dx[t]}{t}=\log(j)-\log(j-1),
$$
valid for $j \ge 2$. 
Summing these inequalities  from $j=k$ to $j=m + k -1 $, we see at once that
\begin{alignat*}{2}
	\log\left(\frac{m+ k}{k} \right)<	\varphi (m)   < \log\left(\frac{m+ k -1}{k-1} \right)
\end{alignat*}	
Ciesielski and Zabcyk relied on this approximation to derive the lower bound of the roots.
\end{rem}
\vskip 5pt
\noindent
We will now introduce the Edmundson-Madansky (EM) inequality, which plays a crucial role in our later arguments
\begin{thm}[EM Inequality]\label{thm.EM}
	Let $X$ be an integrable random variable with support within $[a, b] \subset \mr$. Assume that $f$ is a measurable real function with existing and finite expectation $\mE{f(X)} < \infty$. If $f$ is concave on $[a, b]$, then 
	\begin{equation*}
		\mE{f(X)} - f(a)  \ge (\mE{X}-a)\frac{f(b) - f(a)}{b-a}.
	\end{equation*}	 
\end{thm}
For a comprehensive overview of this inequality, its extensions, and applications, see for instance \cite{Madansky, Gassmann, Donohue}.

\vspace*{5pt}

\noindent
Our next supporting result is an extension of the EM inequality to Poisson random variables, which takes into account the monotonicity of $\varphi$. Recall that the original EM inequality applies to random variables with bounded support. Grassmann and Ziemba (1986) studied a more general case of unbounded random variables but imposed a linear growth condition on the target function. Throughout the rest of this paper, $b(m, p)$ designates the binomial distribution with parameters $m \in \{1, 2, \ldots\}$ and $p \in [0,1]$. 

\begin{thm}\label{lem.upperbound.concave} 
	The function  $\varphi$, defined in (\ref{def.varphi}), satisfies the following inequalities  for a Poisson random variable $X$ with mean $\gamma > 1$
	\begin{equation}\label{exp.varphi.ineq1}
	\varphi(\gamma) \ge  \mE{\varphi(X)} \ge (1-e^{-1}) \varphi(\gamma-1).
		\end{equation}
Specifically, 
	\begin{equation}\label{exp.varphi.ineq2}
\mE{\varphi(X)} \ge  \frac{\mE{Y} }{\gamma+1}\varphi(\gamma),
\end{equation}
where $Y$ is a random variable with distribution 
\begin{equation}\label{def.truncated.poi}
	\prob(Y=k) \coloneqq \begin{dcases}
		\frac{\gamma^ke^{-\gamma}}{k!} &\qquad k \in [1, m+1], \\
		0 &\qquad  k > m+2,\\		
		e^{-\gamma} + \sum_{j =m+1}^{\infty} \frac{\gamma^je^{-\gamma}}{j!} &\qquad k=0,
	\end{dcases}
\end{equation}
having parameter $\gamma$.
\end{thm}
\begin{proof}
For the first inequality, the upper bound is obvious from the concavity and monotonicity of $\varphi$ and the direct application of Jensen's inequality. To establish the lower bound, we proceed in two steps. First, we construct a coupling between $X$ and a suitably chosen Bernoulli random variable that is stochastically smaller. Second,  we apply Edmundson-Madansky inequality to the Binomial variable. Suppose that the mean of $X$ is $\gamma  \in (m, m+1]$, where $m$ is a positive integer. Let $Y_1 = (X_1 \wedge 1) + \cdots + (X_{m} \wedge 1) $ where $X_1, \ldots, X_{m}$ are independent unit-rate Poisson random variables.  Clearly,  $Y_1 \sim \Binomial(m, 1-e^{-1})$ as $(X_1 \wedge 1)$ is Bernoulli with success probability $p=1-e^{-1}$. We want to show that $Y_1 \stochord X$,  where $\stochord$ denotes the usual stochastic order. Instead of using $X$, introduce  $X^\prime \coloneqq X_1+\ldots+X_{m}$, which has mean $m < \lambda$.  In this case, the coupling is between $X^\prime$ and $Y_1$. Given that $X^\prime \stochord X$ (as $m < \lambda$),  it is clear that we need only  verify that $Y_1 \stochord X^\prime$ for $Y_1 \stochord X$ to be true. However,   to prove the former, the necessary and sufficient condition is that $(1-p)^{m} \ge  e^{-\gamma}$ (see \cite{Klenke}). This requirement is clearly met for  $p = 1-e^{-1}$ as $\gamma > m$. 

For the next step, remember  that $\varphi$ is concave according to assertion (\rm{ii}) of Lemma \ref{lem.convex.concave.varphi}. Since  $Y_1$ has support in $[0,m]$, we can apply the EM inequality to obtain	\begin{alignat*}{2}
\mE{\varphi(Y_1)}- \varphi(0)  &\ge (\mE{Y_1}-0)\frac{\varphi(m) - \varphi(0)}{m}.
	\end{alignat*}	
Rearranging the terms and recalling that $\varphi$ is increasing, $\varphi(0)=0$ and $\mE{Y_1} = m(1-e^{-1})$, we arrive at the following inequality after a straightforward manipulation
	\begin{alignat*}{2}
	\mE{\varphi(Y_1)} 
	&\ge  \frac{e-1}{e} \varphi(m) \ge \frac{e-1}{e} \varphi(\gamma-1).
\end{alignat*}	
Finally, use the fact that $\varphi$ is increasing and $X$ is stochastically larger than $Y_1$ to conclude that $\mE{\varphi(X)} \ge \mE{\varphi(Y_1)}$. 

We now turn to the second part of the theorem, let $Y_2$ be a random variable with distribution (\ref{def.truncated.poi}). 
Essentially, the probability mass from the original Poisson distribution that falls outside of  $\{0, 1, \ldots, m +1\}$ is redistributed to the atom $0$. By construction, $Y_2 \stochord X$ is ensured. This can be trivially verified by checking that $\prob(Y_2 > k) \le \prob(X > k) $ for all $k \in \{0, 1, 2, \ldots \}$.  Employing a similar strategy to the concluding steps of the first part, we apply the Edmundson-Madansky inequality to find that
$$\mE{\varphi(X)} \ge  \frac{\mE{Y_2} }{m+1}\varphi(m+1)\ge \frac{\mE{Y_2} }{\gamma+1} \varphi(\gamma).$$%
This concludes the proof.
\end{proof}

\vspace*{5pt}
\noindent
The final supporting result is related to Chebychev's other inequality, published by Andrieve in 1883 \cite{Fink} (see also page 39  in Mitrinovi{\'c} (1970) for additional details). We have adapted this inequality for the scenario involving two stochastically ordered random variables.
\begin{assertion}\label{lem.expectation}
	Let $X, Y: \Omega \to J \subset \mr$ be two independent and integrable random variables defined on a  probability space $(\Omega, {\cal F}, \proba)$. Consider two real-valued and measurable functions $f$ and $g$ defined on $J$, where $\mE{f(X)} < \infty, \mE{f(Y)} < \infty$ and $\mE{g(X)} < \infty, \mE{g(Y)} < \infty$. If $f$ is nondecreasing and $g$ is nonincreasing, then
	\begin{itemize}
		\item [\rm(i)] $\mE{f(X)g(X)} \le \mE{f(X)}\mE{g(X)},$ 
	\end{itemize}
	If in addition  $X \stochord Y$ then
	\begin{itemize}
		\item [\rm(ii)] $\mE{f(X)g(X)} \le \mE{f(Y)}\mE{g(X)},$ and 
		\item [\rm(iii)] $\mE{f(Y)g(Y)} \ge \mE{f(X)}\mE{g(Y)},$ 
	\end{itemize}
\end{assertion}
\begin{proof}
	Assertion {\rm(i)} is the original Chebyshev's other inequality, also called the covariance inequality. We restate the proof of Behboodian (1972) for completeness. Consider the random variable $$Z \coloneqq \left [f(X_1)-f(X_2)\right] \left [g(X_1)-g(X_2)\right], $$
	where $X_1$ and $X_2$ are two independent copies of $X$.  Observe that  $Z \le 0$ almost surely since $f$ is nondecreasing and $g$ is nonincreasing.  Consequently, 
	\begin{alignat*}{2}
		0 & \ge \mE{Z} \\
		&=  \mE{f(X_1)g(X_1)} + \mE{f(X_2)g(X_2)}-\mE{f(X_1)}\mE{g(X_2)} - \mE{f(X_2)}\mE{g(X_1)}\\
		&= 2\mE{f(X)g(X)} - 2\mE{f(X)}\mE{g(X)}.
	\end{alignat*}
	Rearranging the terms in the last equation proves assertion {\rm (i)}.
	
	\vspace*{5pt}
	For the remaining two inequalities, recall that $X \stochord Y$ implies $f(X) \stochord f(Y)$ for $f$ nondecreasing and $g(X) \stochord g(Y)$ for $g$ nonincreasing. Define a new random variable
	\begin{alignat*}{2}
		Y \coloneqq \left[f(X)-f(Y)\right] \left[g(X)-g(Y)\right],
	\end{alignat*} 
	As  $f$ is nondecreasing, $f(X) \stochord f(Y)$. Conversely, $g(Y) \stochord g(X)$ due to the nonincreasing nature of $g$. Together, these imply that $Y\le 0$ almost surely.
	Next, compute the expectation of $Y$ to see that
	\begin{alignat*}{2}
		0 \ge \mE{Y} &=  \mE{f(X)g(X)} + \mE{f(Y)g(Y)}-\mE{f(X)}\mE{g(Y)} - \mE{f(Y)}\mE{g(X)}\\
		&\ge   \mE{f(X)g(X)}  +    \underbrace{\mE{f(Y)g(Y)}-\mE{f(Y)}\mE{g(Y)}}_{\ge 0 \text{ by part } {\rm(i)}} - \mE{f(Y)}\mE{g(X)}\\
		&\ge   \mE{f(X)g(X)} - \mE{f(Y)}\mE{g(X)} 
	\end{alignat*}
	For part {\rm(iii)}, we proceed in much the same way by  substituting $\mE{f(X)}\mE{g(X)}$ for $\mE{f(Y)}\mE{g(X)}$ and the proof is complete.
\end{proof}

\vspace*{0.25cm}
\noindent
We are now in a position to establish the main theorem of this section.
\begin{proof}[Proof of Theorem \ref{thm.bounds.roots}]
	Recall the  derivative of the Pochhammer symbol of a positive real $x$, which is given by
	$$(x)_n = (x)_n\left\{\psi(x+n) - \psi(x)\right\}.$$
	Applying this definition, the derivative in the first parameter of $M(k, k+\theta, x)$ becomes
	\begin{alignat*}{2}
		D_a M(k, k+\theta; x) &= \sum_{n \ge 0} \frac{x^n}{n!} \frac{(k)_n}{(k+\theta)_n}\left\{  \psi(k+n) -\psi(k)  \right\} \nonumberj
		&= \sum_{n \ge 0} \frac{x^n}{n!}   \phi(n)\varphi(n).
	\end{alignat*}
	As a consequence, we may rewrite  equation (\ref{root.eqn.poi}) as $h_k(x) = 0$, where
	\begin{alignat*}{2}
		h_k(x) &= {\cal S}_1(1-x/\lambda,k) - {\cal S}_0(1-x/\lambda,k)\nonumberj
		&=\sum_{n \ge 0}\frac{x^n}{n!}e^{-x}\phi(n)(\theta\varphi(n)-1)\nonumberj
		&= \mE{ \phi(\xi_x) \left\{ \theta \varphi(\xi_x) -1\right\}}, 
	\end{alignat*}
	with $\xi_x$ representing a Poisson random variable with mean $x$. 
	
	Let $\gamma_k$  be the unique root of $h_k(x)=0$.   Under this condition,  $h_k(x)=0$ can be expressed as
	\begin{equation} \label{h.eqn2}
		\frac{1}{\theta} \mE{\phi(\xi_{\gamma_k}) } = \mE{ \phi(\xi_{\gamma_k})\varphi(\xi_{\gamma_k})}.
	\end{equation}
	Given that for any prescribed  $\epsilon \in [0, \gamma_{k})$, $\xi_{\gamma_{k}-\epsilon} \stochord \xi_{\gamma_k} \stochord \xi_{\gamma_{k}+\epsilon} $, we can apply assertion (iii) of Proposition (\ref{lem.expectation}) to the right-hand side of (\ref{h.eqn2}) to determine  that
	\begin{alignat*}{2}
		\frac{1}{\theta} \mE{\phi(\xi_{\gamma_k}) } 
		&\ge \mE{ \phi(\xi_{\gamma_{k}})}\mE{\varphi(\xi_{\gamma_{k}-\epsilon})}. 
	\end{alignat*}
	Therefore, 
	$\displaystyle\mE{\phi(\xi_{\gamma_k-\epsilon}) } \le \frac{1}{\theta}.$
	Repeating this process and applying assertion  {\rm(ii)} of the same Proposition results in 
	$\displaystyle\mE{\varphi(\xi_{\gamma_k+\epsilon}) } \ge \frac{1}{\theta}.$ Taking the limit  as $\epsilon \to 0$, it becomes evident that  $$\mE{ \varphi(\xi_{\gamma_{k}})} = \frac{1}{\theta}.$$
	Subsequently, we apply (\ref{exp.varphi.ineq1})   to bound $\mE{ \varphi(\xi_{\gamma_{k}})}$, leading to the fundamental relationship
	$$ \frac{e }{e-1}\varphi(\gamma_k-1) \le  \frac{1}{\theta} \le  \varphi(\gamma_k).$$
	
	Note that we did not require the assumption $\gamma_k > 1$, as the right-hand side of the inequality (ensured by Jensen's inequality) guarantees that $\gamma_k > 1$ for all $k \ge 2$. 
	Whence, to find bounds for $\gamma_k$, it is sufficient to consider (\ref{varphi.ineq}) which gives
	\begin{alignat*}{2}
		\frac{e -1 }{e} \log\left(\frac{\gamma_k + k-1}{k-1} \right)  \le \frac{1}{\theta} \le \log\left(\frac{\gamma_k+ k-1}{k-1} \right) - \frac{1}{2(k-1)}. 
	\end{alignat*}		
	Performing elementary calculations yields the desired inequality (\ref{poisson.root.bounds}).
	
	Moving on to the second part of the Theorem, first, we apply (\ref{exp.varphi.ineq2}) 
	to $\mE{ \phi(\xi_{\gamma_{k}})}$, which leads to
	$$ \frac{\mE{Y_{\gamma_k}}}{\gamma_k+1}\varphi(\gamma_k) \le  \frac{1}{\theta} \le  \varphi(\gamma_k), $$
	where $Y_{\gamma_k}$ is a random variable with distribution (\ref{def.truncated.poi}) and parameter $\gamma_k$.  Repeating similar steps used at the end of the first part, we obtain
	\begin{alignat*}{2}
		(e^{\textstyle\frac{1}{\theta}+\frac{1}{2(k-1)}}-1)(k-1) \le \gamma_k  \le (e^{\textstyle\frac{1}{\theta}\,\frac{\gamma_k+1}{\mE{Y_{\gamma_k}}}}-1)k-1.		
	\end{alignat*}
	Finally, divide each side of this inequality by $k$ then observe that as  $k \to \infty$, $\gamma_k \to \infty$,  $Y_{\gamma_k} \to \xi_{\gamma_k}$ hence, $\frac{\gamma_k+1}{\mE{Y_{\gamma_k}}} =  \frac{\gamma_k+1}{\gamma_k}\to 1$. Consequently, formula (\ref{exp.varphi.ineq2}) is established.
\end{proof}
\begin{rem}
	In the scenario where $\theta=1$, aside from the works by \cite{CZ, Ciesielski} and \cite{BG}, Bruss (1987) derived the same limit using a probabilistic method. This argument similarly applies to the Karamata-Stirling profile to obtain the limit $e^{1/\theta}-1$.
\end{rem}
\subsection{Numerical results}
The table below presents numerical calculations of lower and upper bounds for roots. The improvement factor for each bound is the ratio of the new bound to the old bound, calculated in the correct order. The overall range improvement factor is the ratio of the initial range (upper bound - lower bound) to the new range.
%
%
%

\begin{table}[H]
	\centering
	\begin{tabular}{cccccccc}
		\toprule
		\textbf{Root number} & \multicolumn{2}{c}{\textbf{Original bounds}} & \multicolumn{2}{c}{\textbf{New bounds}} & \multicolumn{3}{c}{\textbf{Improvement Factor}}  \\[5pt]
		$k$ & Lower & Upper & Lower & Upper & Lower & Upper & Range\\
		\midrule
2 & 1.72 & 48.24 & 2.79 & 3.86 & 1.63 & 12.48 & 43.44 \\
3 & 3.44 & 66.92 & 4.64 & 7.73 & 1.35 & 8.66 & 20.55 \\
4 & 5.15 & 85.61 & 6.41 & 11.59 & 1.24 & 7.38 & 15.51 \\
5 & 6.87 & 104.29 & 8.15 & 15.46 & 1.19 & 6.75 & 13.33 \\
6 & 8.59 & 122.97 & 9.88 & 19.32 & 1.15 & 6.36 & 12.12 \\
7 & 10.31 & 141.65 & 11.61 & 23.19 & 1.13 & 6.11 & 11.35 \\
8 & 12.03 & 160.34 & 13.34 & 27.05 & 1.11 & 5.93 & 10.82 \\
9 & 13.75 & 179.02 & 15.06 & 30.92 & 1.10 & 5.79 & 10.43 \\
10 & 15.46 & 197.70 & 16.79 & 34.78 & 1.09 & 5.68 & 10.13 \\
100 & 170.11 & 1879.18 & 171.47 & 382.59 & 1.01 & 4.91 & 8.10 \\
1000 & 1716.56 & 18693.97 & 1717.92 & 3860.70 & 1.00 & 4.84 & 7.92 \\
		\bottomrule
	\end{tabular}
	\caption{Comparison of Numerical Bounds: Ciesielski-Zabczyk vs. New Bounds.}
\end{table}

\newpage
\appendix
\section*{Appendix}
\begin{proof}[Proof of part \rm(ii) of Theorem \ref{kummer.signn}]
The logarithmic derivative in parameter $c$,	where  $a_1, a_2$ are  constants satisfying  $c > a_1 > 0$ and $c > a_2 > 0$ respectively is given by
\begin{alignat}{2}\label{proofkummer}
	D_c \displaystyle \log\left( \frac{M(a_1, c, x)}{M(a_1, c, x)} \right) &= D_c\log\frac{\Gamma(c-a_2)}{\Gamma(c-a_1)} + \log D_c \frac{\phi(a,c_1,x)}{\phi(a,c_2,x)}.
\end{alignat}
By the elementary Lemma \ref{lem.monotonicity.log},  it is sufficient to show that the signs of the two derivatives of quotients match. 
First, the sign of the derivative of a quotient of two Gamma functions  coincides with $a_1 - a_2$  based on Lemma \ref{lem.gamma.sign}. 

Next, we consider the second logarithmic derivative involving a quotient of $\phi$'s. Use Lemma \ref{lem.monotonicity.log} once more to switch back to  the derivative of the quotient  $\displaystyle {\phi(a_1,c,x)}/{\phi(a_2,c,x)}$. Given that the denominator in the expression of the derivative of a quotient is positive, we restrict our investigation to the numerator, which we denote $\phi_c$.

In much the same way as in the proof of part {\rm(i)}, consider the Eulerian integral representation (\ref{kummer.integral2}) to express $\phi_c$ as a double integral, 
where $c > a_1 > 0$ and $c > a_2 > 0$ respectively
\begin{alignat*}{2}
	&\phi_c \coloneqq  -\int_{0}^1 e^{xu}(1-u)^{c-2}\left( \frac{u}{1-u}\right)^{a_1-1} \log (1-u)\dx[u]\int_{0}^1 e^{xv}(1-v)^{c-2} \left( \frac{v}{1-v}\right)^{a_2-1}  \dx[v]
	\\&\qquad +		\int_{0}^1 e^{xu}(1-u)^{c-2}\left( \frac{u}{1-u}\right)^{a_1-1} \dx[u]\int_{0}^1 e^{xv}(1-v)^{c-2} \left( \frac{v}{1-v}\right)^{a_2-1}  \log (1-v)\dx[v]\\
	&=\iint_{[0,1]\times[0,1]} e^{x(u+v)}[(1-u)(1-v)]^{c-2}\left( \frac{uv}{(1-u)(1-v)}\right)^{a_1-1} \left( \frac{v}{1-v}\right)^{a_2-a_1} \log (1-v)\dx[u] \dx[v]
	\\&\qquad -\iint_{[0,1]\times[0,1]} e^{x(u+v)}[(1-u)(1-v)]^{c-2}\left( \frac{uv}{(1-u)(1-v)}\right)^{a_1-1} \left( \frac{v}{1-v}\right)^{a_2-a_1} \log (1-u)\dx[u] \dx[v]
\end{alignat*}			
\begin{alignat*}{2}		
	&=\iint_{[0,1]\times[0,1]} \left\{ e^{x(u+v)}[(1-u)(1-v)]^{c-2}\left( \frac{uv}{(1-u)(1-v)}\right)^{a_1-1}\right\} \left( \frac{v}{1-v}\right)^{a_2-a_1} \log \left(\frac{1-v}{1-u}\right)\dx[u] \dx[v]\\
	&=\iint_{0 < v < u < 1} \left\{ \cdots\right\} \left( \frac{v}{1-v}\right)^{a_2-a_1} \log \left(\frac{1-v}{1-u}\right)\dx[u] \dx[v] + \iint_{0 < u < v < 1} \left\{ \cdots\right\} \left( \frac{v}{1-v}\right)^{a_2-a_1} \log \left(\frac{1-v}{1-u}\right)\dx[u] \dx[v]\\
	&=\iint_{0 < v < u < 1} \left\{ \cdots\right\} \left( \frac{v}{1-v}\right)^{a_2-a_1} \log \left(\frac{1-v}{1-u}\right)\dx[u] \dx[v] + \iint_{0 < v < u < 1} \left\{ \cdots\right\} \left( \frac{u}{1-u}\right)^{a_2-a_1} \log \left(\frac{1-u}{1-v}\right)\dx[u] \dx[v]\\
	&=\iint_{0 < v < u < 1} \left\{ \cdots\right\}\left\{ \left( \frac{v}{1-v}\right)^{a_2-a_1} - \left( \frac{u}{1-u}\right)^{a_2-a_1} \right\}  \log \left(\frac{1-v}{1-u}\right)\dx[u] \dx[v].
\end{alignat*}
Mirroring the case of the derivative in parameter $a$, the term  $\displaystyle\log\left( \frac{1-v}{1-u} \right)$ is positive in the domain of integration, where $u > v$. It follows that the sign of the last integrand coincides with the sign of $$\displaystyle\left( \frac{v}{1-v}\right)^{a_2-a_1} - \left( \frac{u}{1-u}\right)^{a_2-a_1},$$ whose sign precisely matches that of $c_1 - c_2$. By the second part of Lemma \ref{lem.monotonicity.log}, this concludes the proof as the sign of both terms in the right-hand side of (\ref{proofkummer}) matches  $c_1 - c_2$. 
\end{proof}

\end{document}